\documentclass[jip]{degruyter-journal-a}      % J. Inverse Ill-Posed Probl.
\usepackage{algorithm}
\usepackage{algorithmic}

\renewcommand{\leq}{\leqslant}
\renewcommand{\geq}{\geqslant}
\newcommand{\argmin}{\operatornamewithlimits{argmin}}
\newcommand{\argmax}{\operatornamewithlimits{argmax}}

\newcommand{\mc}[1]{\mathbb{\item1}}
\newcommand{\circled}[1]{\raisebox{.5pt}{\textcircled{\raisebox{-.9pt} {#1}}}}
\def\la{\langle}
\def\ra{\rangle}
\def\R{\mathbb{R}}

\DeclareMathOperator{\spn}{span}

\DeclareMathOperator{\prox}{prox}
\DeclareMathOperator{\Step}{Step}

\title{Alternating Minimization Methods for Strongly Convex Optimization}
\headlinetitle{Alternating Minimization Methods for Strongly Convex Optimization}
\headlineauthor{N. Tupitsa, P. Dvurechensky, A. Gasnikov and  S. Guminov}

\lastnameone{Tupitsa}
\firstnameone{Nazarii}
\nameshortone{}
\addressone{Moscow Institute of Physics and Technology, Dolgoprudny, Moscow, Russia; \\ Institute for Information Transmission Problems RAS, Moscow, Russia;\\ National Research University Higher School of Economics, Moscow}
\countryone{Russia}
\emailone{tupitsa@phystech.edu}

\lastnametwo{Dvurechensky}
\firstnametwo{Pavel}
\nameshorttwo{}
\addresstwo{Weierstrass Institute for Applied Analysis and Stochastics, Berlin;\\ 
Institute for Information Transmission Problems RAS, Moscow}
\countrytwo{Russia}
\emailtwo{pavel.dvurechensky@wias-berlin.de}

\lastnamethree{Gasnikov}
\firstnamethree{Alexander}
\nameshortthree{}
\addressthree{National Research University Higher School of Economics, Russia;  \\Moscow Institute of Physics and Technology, Dolgoprudny, Russia; \\ Institute for Information Transmission Problems RAS, Moscow}
\countrythree{Russia}
\emailthree{gasnikov@yandex.ru}

\lastnamefour{Guminov}
\firstnamefour{Sergey}
\nameshortfour{Russia}
\addressfour{Moscow Institute of Physics and Technology, Dolgoprudny, Russia; \\ Institute for Information Transmission Problems RAS, Moscow}
\countryfour{}
\emailfour{sergey.guminov@phystech.edu}

\lastnamefive{}
\firstnamefive{}
\nameshortfive{}
\addressfive{}
\countryfive{}
\emailfive{}

\abstract{We consider alternating minimization procedures for convex optimization problems with variable divided in many    block, each block being amenable for minimization with respect to its variable with freezed other variables blocks. In the case of two blocks, we prove a linear convergence rate for alternating minimization procedure under Polyak-Łojasiewicz condition, which can be seen as a relaxation of the strong convexity assumption. Under strong convexity assumption in many-blocks setting we provide an accelerated alternating minimization procedure with linear rate depending on the square root of the condition number as opposed to condition number for the non-accelerated method.
We also mention an approximating non-negative solution to a linear system of equations $Ax=y$ with alternating minimization of Kullback-Leibler (KL) divergence between $Ax$ and $y$.
}

\keywords{convex optimization, alternating minimization, block-coordinate method, complexity analysis}

\classification{	90C25}

\researchsupported{This research was funded by Russian Science Foundation (project 18-71-10108) and by the Ministry of Science and Higher Education of the Russian Federation (Goszadaniye) N\textsuperscript{\underline{o}}075-00337-20-03, project No. 0714-2020-0005.}
\acknowledgments{}

\dedicated{}

\begin{document}

\section{Introduction}
In this paper we consider the minimization problem
\begin{equation}
    \label{eq:pr_st}
    \min_{x \in Q \subset \R^m} f(x), 
\end{equation}
where $f(x)$ is a smooth convex function with $L$-Lipschitz-continuous gradient.
Further, our main assumption is that the space $\R^m$ can be divided into $n$ disjoint subspaces $ \mathcal L_i \in \R^{n_i}$, $\sum n_i = m$, s.t. $\cup \mathcal L_i = \R^m$ and it is possible to minimize the objective $f$ in each block if the variables in all other blocks are fixed. Moreover, we are mostly interested in obtaining linear convergence rate and sufficient conditions for it.

To be exact, we suppose that $f$ has a block structure, i.e. $f(x) = f(x_1, \dots, x_n)$, and we know exact expression for the minimizer over $i$-th block
\[x^*_i = \argmin_{z\in Q_i \subset \mathcal L_i} f(x_1, \dots, x_{i-1}, z, x_{i+1}, \dots, x_n).\]
where $\cup Q_i = Q$. %\pd{Connection between $Q$ and $Q_i$ should be defined.}\na{done}

A very old and natural idea under this assumption is to use alternating minimization procedure \cite{ortega2000iterative,bertsekas1989parallel}, where the objective is minimized sequentially in each block. First of all, we are interested in the convergence rate analysis of this type of algorithms. For smooth strongly convex problems under some additional technical assumptions, the linear rate was obtained in \cite{luo1993error}. In \cite{beck2015convergence} the author analyze alternating minimization procedure for the case of two blocks in the general convex setting. The underlying assumption is presence of a smooth component in at least one block of variables. Also non-smoothness is possible via composite terms which still allow the block minimization. Since there is no strong convexity assumption, the obtained convergence rate is sublinear, namely $O(1/k)$, where $k$ is the iteration counter. Similar result, but for many-block setting was obtained in \cite{hong2017iteration,sun2015improved}. In the fully smooth setting under strong convexity assumption \cite{nutini2015coordinate} obtain linear rate of convergence also for the many-block setting. This linear rate is proportional to $\kappa$ -- efficient condition number of the problem. The autors of \cite{chambolle2017accelerated} provide an accelerated alternating minimization method for a very special problem with two blocks having the form of a sum of a quadratic function with two proximally friendly composite terms. The obtained convergence rate is $O(1/k^2)$ for convex setting and is linear with exponent $\sqrt{\kappa}$ in the strongly convex case. The authors of \cite{diakonikolas2018alternating} analyze a non-accelerated alternating minimization method and obtain $O(1/k)$ convergence rate in the convex setting and linear rate with exponent $\kappa$ for strongly convex case. They also propose an accelerated method for general convex setting with rate $O(1/k^2)$ and conjecture that their analysis can be extended for the strongly convex case. Interested readers can look also into the review \cite{hong2016unified}.

In this paper we, firstly, focus on obtaining linear rate of convergence for non-accelerated method with the exponent $\kappa$ in a more general setting of Polyak-Łojasiewicz condition \cite{polyak1987introduction}. This assumption is weaker than the strong convexity assumption since it follows from the strong convexity. Secondly, we propose an accelerated alternating minimization method for general smooth objective functions in the many-blocks setting. For this method we obtain accelerated convergence rate 
\[
O\left( \min\left\{ \frac{1}{k^2}, \left(1-\sqrt{\kappa}\right)^k\right\}\right).
\]

From the perspective of applications, many existing statistical algorithms %that minimize certain combinations of Kullback–Leibler (KL) divergence, and 
can be derived as alternating minimization of Kullback–Leibler (KL) divergence~\cite{Csiszr1984InformationGA}. 
These include the expectation maximization (EM) algorithm for likelihood maximization~\cite{vardi1985statistical, andresen2016convergence}, the Bayesian maximum \textit{a posteriori} (MAP) method with gamma-distributed priors~\cite{lange1987theoretical}, the multiplicative algebraic reconstruction technique \sloppy (MART)  \cite{GORDON1970algebraic} and the "simultaneous" MART (SMART) algorithm~\cite{byrne1992iterative}. Each of these algorithms can be viewed as an algorithm to find an approximate non-negative solution to a linear system $Ax=y$. For example, the SMART can be shown to minimize KL(Ax, y)~\cite{byrne1992iterative}.
Some other application of optimization to inverse problems can be found in~\cite{ye2019optimization, vogel2002computational, byrne2014iterative}.

Another example is a system of nonlinear equations $g(x)=0$, where $g:\R^n \rightarrow \R^m$, $m<n$ and there exists some $\mu$ s.t. for any $x\in\R^n$
\[
\lambda_{min}\left( \frac{\partial g(x)}{\partial x} \left[\frac{\partial g(x)}{\partial x}\right]^T\right) \geq \mu.
\] 
Then the function $f(x) = \|g(x)\|_2^2$ satisfies Polyak-Łojasiewicz condition~\cite{nesterov2006cubic}, and the algorithms analyzed here can be applied and have linear convergence rate.

\section{Simple alternating minimization algorithm and notation}
Consider for simplicity alternating minimization algorithm for the problem with only two block structure. All the following results and proofs can be easily extended for any number of blocks.

Consider alternating minimization Algorithm~\ref{AM} for the problem %\pd{Assumptions on $f,g$ should be stated explicitly.}\na{done}
\begin{equation}
    \label{F}
     \min_{x_1 \in Q_1, x_2 \in Q_2} F(x_1,x_2) \equiv f(x_1,x_2)+g_1(x_1)+g_2(x_2),
\end{equation}
where $f(x)$ is a smooth convex function with $L$-Lipschitz-continuous gradient and each $g_i(x)$ is a convex possibly non-smooth function.

\begin{algorithm}[H]
\caption{Alternating Minimization (AM)}
\label{AM}
\begin{algorithmic}[1]
   \REQUIRE Starting point $x_0$.
    \ENSURE $x^k$
    \STATE Set $x^0$.
    \FOR{$k \geqslant 0$}
    \IF  {$k$ mod $2=0$}
        \STATE $x_1^{k+1} = \argmin_{z \in Q_1} f(z, x_2^k) + g_1(z)$
	\ELSE
	    \STATE $x_2^{k+1} = \argmin_{z \in Q_2} f(x_1^{k+1}, z) + g_2(z)$
	\ENDIF
\ENDFOR
\end{algorithmic}
\label{alg-1}
\end{algorithm}
We introduce the following notation:
\[ x^k = (x^k_1, x^k_2), \quad  x^{k+\frac{1}{2}} = (x^{k+1}_1, x^k_2)\]
\[ T_{M}(x) = (T^1_{M}(x), T^2_{M}(x)) \quad G_{M}(x) = (G_{M}^1(x), G_{M}^2(x))\]
\begin{equation}
    T^i_{M}(x) = \prox_{\frac{1}{M}g_i}\left(x_i - \frac{1}{M}\nabla_i f(x)\right), \quad G_{M}^i(x) = M(x_i - T_{M}^i(x)) 
    \label{T}
\end{equation}

For the case $i=1$
\begin{multline*}
    T_{M}^1(x^k) = \argmin_{u \in Q_1}\left(g_1(u) + \frac{{M}}{2} \|u-(x^k_1 - \frac{1}{{M}}\nabla_1 f(x^k))\|^2\right) =
    \\
    = \argmin_{u \in Q_1}\left(g_1(u) + \frac{{M}}{2} \|u-x^k_1\|^2 + \la \nabla_1 f(x^k), u - x^k_1\ra\right).
\end{multline*}

Next we write
\begin{align*}
    \partial_1 F(x_1^{k+1}, x_2^k) &= \nabla_1 f(x_1^{k+1}, x_2^k) + \partial g_1(x_1^{k+1})
    \\
    \partial_2 F(x_1^k, x_2^k) &= \nabla_2 f(x_1^k, x_2^k) + \partial g_2(x_2^k),
\end{align*}
where $\partial_1 F(x_1^{k+1}, x_2^k)$ denotes a subgradient of $F$ w.r.t first block, e.g. such a set $S$, that for all $s \in S$ the following holds
\[
    F(y, x_2^k) \geq F(x_1^{k+1}, x_2^k) + \la s, y - x_1^{k+1} \ra.
\]
$\partial_2 F(x_1^k, x_2^k)$ is defined similarly.

Then, optimality conditions can be written as follows:
\begin{align}
    \la \nabla_1 f(x_1^{k+1}, x_2^k), u - x_1^{k+1} \ra &\geq \la -\partial g_1(x_1^{k+1}), u - x_1^{k+1} \ra \notag
    \\
    \la \nabla_2 f(x_1^k, x_2^k), v - x_2^{k} \ra &\geq \la -\partial g_2(x_2^k), v - x_2^{k} \ra
    \label{opt-0.1}
\end{align}
for all $u \in Q_1$, $v \in Q_2$.

The following should clarify the notation.
%\pd{Some text of the Lemma statement should be present.}
\begin{lemma}
    For points, generated by Algorithm~\ref{AM} the following holds
    \[G^1_{M}(x^{k+\frac{1}{2}}) = 0, \quad G^2_{M}(x^k) = 0\]
    \[T^2_{M}(x^k) = x_2^k, \quad T^1_{M}(x^{k+\frac{1}{2}}) = x_1^{k+1}\] for all $k$.
\end{lemma}
\begin{proof}
    \[T^2_{M}(x^k) = \argmin_{v \in Q_2}\left(g_2(v) + \frac{{M}}{2} \|v-x^k_2\|^2 + \la \nabla_2 f(x^k), v - x^k_2\ra\right) = x_2^k,\]
    The first is that $\la \partial g_2(x_2^k) + \nabla_2 f(x_1^k, x_2^k), v - x_2^{k} \ra \geq 0$ for all $v \in Q_2$ by (\ref{opt-0.1}), second $\la \partial \|v-x^k_2\|^2; v-x^k_2\ra \geq 0$ since $x^k_2$ is a minimizer of $\|v-x^k_2\|^2$. Summing this two inequalities implies optimality condition  for $T_M^2(x^k)$ at the point $x_2^k$.
    \[G_{M}^2(x^k) = M(x^k_2 - T_{M}^2(x^k)) = M(x^k_2 - x^k_2) = 0,\] where the last equality follows from the definition of $G^2_{M}(x^k)$.
\end{proof}

Introduce also the following notation:
\begin{multline}
    \label{block-prox-pl}
    \mathcal{D}_1(x^k,M) 
    \\
    \equiv -2M \min_{u \in Q_1} \big[ \langle \nabla_1 f(x^k) , u-x_1^k \rangle + \frac{M}{2}||u-x_1^k||^2+ g_1(u) - g_1(x_1^k)  \big]
\end{multline}
\begin{multline}
    \mathcal{D}_2(x^{k+\frac{1}{2}},M) 
    \\
    \equiv -2M \min_{v \in Q_2} \big[ \langle \nabla_2 f(x^{k+\frac{1}{2}}), v-x_2^k \rangle + \frac{M}{2}||v-x_2^k||^2 + g_2(v) - g_2(x_2^k) \big].
\end{multline}
Notice that $T_{M}(x^k)$ and $T_{M}(x^{k+\frac{1}{2}})$ are corresponding minimizers of these two above problems.

\section{Proximal Polyak-Łojasiewicz condition}\label{app:probs}
In this section we prove that strongly convex function satisfies the proximal-PL inequality condition \cite{karimi2016linear}. %\pd{$\mathcal{D}$ is not defined.} \na{it's defined in \eqref{block-prox-pl}, isn't it?}

We suppose, that strong convexity parameter can different for different variable blocks:
\begin{multline*}
    f(u, v) \geqslant f(\xi, \eta) + \langle \nabla_1 f(\xi,\eta), u-\xi \rangle + \langle \nabla_2 f(\xi, \eta), v-\eta \rangle
    \\
    +  \frac{\mu_1}{2} \|u-\xi\|^2 + \frac{\mu_2}{2} \|v-\eta\|^2,
\end{multline*}
for any $u,\xi \in Q_1$ and $v,\eta \in Q_2$.
Notice, that the single variable definition can be written with $\mu = \min(\mu_1, \mu_2)$.

%\pd{It is hard to understand what is happening in what follows. What does this enumeration mean? Cases?}

The main result of this section reads as follows.
\begin{theorem}\label{th:prox-pl}
 If $f$ is strongly convex and $g$ is convex, then $F(x) = f(x)  + g(x)$ satisfies proximal PL-conditions
\begin{equation}
    \label{prox-pl}
    F^*  \geq  F(x^k) - \frac{1 }{ 2\mu_1}\mathcal{D}_1(x^k,\mu_1), \quad
    F^*  \geq  F(x^{k+\frac{1}{2}}) - \frac{1}{2\mu_2}\mathcal{D}_2(x^{k+\frac{1}{2}},\mu_2),
\end{equation}

for the points $x^k$ and $x^{k+\frac{1}{2}}$, generated by Algorithm~\ref{AM}.
\end{theorem}

\begin{proof}
By the strong convexity of $f$ we have
\begin{multline*}
    f(u, v) \geqslant f(x^k_1, x^k_2) + \langle \nabla_1 f(x^k_1, x^k_2), u-x^k_1 \rangle + \langle \nabla_2 f(x^k_1, x^k_2), v-x^k_2 \rangle
    \\
    \shoveright{+  \frac{\mu_1}{2} \|u-x^k_1\|^2 + \frac{\mu_2}{2} \|v-x^k_2\|^2 \stackrel{\scriptsize{\circled{1}}}{\geq}}
    \\
    \stackrel{\scriptsize{\circled{1}}}{\geq} f(x^k_1, x^k_2) 
    + \langle \nabla_1 f(x^k_1, x^k_2), u-x^k_1 \rangle - \langle \partial g_2( x^k_2), v-x^k_2 \rangle +  \frac{\mu_1}{2} \|u-x^k_1\|^2 \stackrel{\scriptsize{\circled{2}}}{\geq}
    \\
    \stackrel{\scriptsize{\circled{2}}}{\geq} f(x^k_1, x^k_2) 
    + \langle \nabla_1 f(x^k_1, x^k_2), u-x^k_1 \rangle + g_2(x^k_2) + g_2(v) +  \frac{\mu_1}{2} \|u-x^k_1\|^2
\end{multline*}
where
\begin{itemize}
    \item \circled{1} by ~\eqref{opt-0.1}
    \item \circled{2} by convexity of $g_2$
\end{itemize}

which leads to
\begin{multline*}
    F(u, v) 
    \\
    \shoveright{\quad\geqslant F(x^k_1, x^k_2) + \langle \nabla_1 f(x^k_1, x^k_2), u-x^k_1 \rangle - g_1(x^k_1) + g_1(u) +  \frac{\mu_1}{2} \|u-x^k_1\|^2.}
\end{multline*}
Minimizing both sides respect to $u \in Q_1, v\in Q_2$, 
\begin{multline}
    F^*  \geq F(x^k) + \min_{u} \big[\langle \nabla_1 f(x^k) , u-x_1^k \rangle + \frac{\mu_1}{2}||u-x_1^k||^2 + g_1(u) - g_1(x_1^k)\big]  
    \\
    =  F(x^k) - \frac{1 }{ 2\mu_1}\mathcal{D}_1(x^k,\mu_1).
\end{multline}
Rearranging, we have our result.

The similar result holds for the point $x^{k+\frac{1}{2}}$:
\[
F^*  \geq  F(x^{k+\frac{1}{2}}) - \frac{1 }{ 2\mu_2}\mathcal{D}_2(x^{k+\frac{1}{2}},\mu_2).
\]

\end{proof}

We also need the Corollary 1 from~\cite{karimi2016linear}, which proof is almost the same as the proof of the following lemma:
\begin{lemma}\label{lem:1}
For any differentiable function $f$ and any convex function $g$, given $\lambda_2 > \lambda_1 > 0$ we have
\begin{align*}
    \mathcal{D}_1(x^k,\lambda_2) &\geq  \mathcal{D}_1(x^k,\lambda_1),
    \\
    \mathcal{D}_2(x^{k+\frac{1}{2}},\lambda_2) &\geq  \mathcal{D}_2(x^{k+\frac{1}{2}},\lambda_2).
\end{align*}
\end{lemma}

\begin{proof}
By convexity of $g$ for $0<\alpha<1$ 
\begin{align*}
    g(\alpha z) = g(\alpha z + (1-\alpha) \cdot 0) \leq \alpha g(z) + (1-\alpha) g(0).
\end{align*}
Then with $z=\frac{\zeta}{\lambda_1}$ and $\alpha=\frac{\lambda_1}{\lambda_2}$

\begin{align*}
    g\left(\frac{\zeta}{\lambda_2}\right) - g(0)  &\leq \frac{\lambda_1}{\lambda_2} \big( g\left(\frac{\zeta}{\lambda_1}\right) - g(0)\big)
    \\
    \lambda_2 \cdot \Bigg(g\left(\frac{\zeta}{\lambda_2}\right) - g(0)\Bigg) &\leq \lambda_1 \cdot\Bigg(g\left(\frac{\zeta}{\lambda_1}\right) - g(0)\Bigg)
\end{align*}
Then move values of our function:

\begin{align*}
    \lambda_2 \cdot \Bigg(g\left(\frac{\zeta}{\lambda_2}+x_1^k\right) - g(0+x_1^k)\Bigg) &\leq \lambda_1 \cdot\Bigg(g\left(\frac{\zeta}{\lambda_1}+x_1^k\right) - g(0+x_1^k)\Bigg)
\end{align*}

and add to both sides
\[
    h(\zeta) = \langle \nabla_1 f(x^k) , \zeta \rangle + \frac{1}{2}||\zeta||^2 
\]
we have

\begin{multline*}
    \min_{\zeta \in  Q} \langle \nabla_1 f(x^k) , \zeta \rangle + \frac{1}{2}||\zeta||^2 + \lambda_2 \cdot \Bigg(g\left(\frac{\zeta}{\lambda_2}+x_1^k\right) - g(x_1^k)\Bigg)
    \\
    \leq \min_{\zeta \in Q} \langle \nabla_1 f(x^k) , \zeta \rangle + \frac{1}{2}||\zeta||^2 + \lambda_1 \cdot\Bigg(g\left(\frac{\zeta}{\lambda_1}+x_1^k\right) - g(x_1^k)\Bigg)
\end{multline*}

Or with the change of variables $\zeta = \lambda_i(u-x_1^k)$

\begin{multline*}
    \lambda_2 \min_{u \in \frac{Q}{\lambda_2} + x_1^k} \langle \nabla_1 f(x^k) , u-x_1^k \rangle + \frac{\lambda_2}{2}||u-x_1^k||^2 + g\left(u\right) - g(x_1^k)
    \\
    \leq \lambda_1 \min_{u \in \frac{Q}{\lambda_1} + x_1^k} \langle \nabla_1 f(x^k) , u-x_1^k \rangle + \frac{\lambda_1}{2}||u-x_1^k||^2 + g\left(u\right) - g(x_1^k)
\end{multline*}

which holds if 
\begin{equation}
    \label{sets-structure}
    \frac{Q}{\lambda_2} + x_1^k \subset \frac{Q}{\lambda_1} + x_1^k.
\end{equation}
For example it holds if $Q = \R^n$.

\end{proof}

\section{Convergence}
In this section we prove convergence rate of Algorithm~\ref{AM}.
If proximal PL-condition hold for $F$, then one can guarantee linear convergence rate, if not, the convergence is polynomial. The two following subsections contain proofs of that.
\subsection{Linear convergence}
Lipschitz continuity of the gradient of function $f$ w.r.t. a $\|\cdot\|$ implies
\begin{multline}
    f(u, v) \leqslant f(\xi, \eta)
    + \langle \nabla_1 f(\xi,\eta), u-\xi \rangle + \langle \nabla_2 f(\xi, \eta), v-\eta \rangle 
    \\
    + \frac{L_1}{2} \|u-\xi\|^2 + \frac{L_2}{2} \|v-\eta\|^2,
\end{multline}
where again we suppose that constant $L_1$ and $L_2$ can be different for different blocks,
and the the constant in the regular definition of Lipshitz continuity of the gradient of $f$ is described by $L = \max(L_1, L_2)$.

\begin{theorem}\label{th:linear}
If $F$ from \eqref{F} satisfies the proximal-PL inequality ~\eqref{block-prox-pl}. Then the algorithm \ref{alg-1} has a linear convergence. %\pd{$L_1,L_2$ are not defined.}\na{done}
\begin{equation}
    F(x^{k+1}) - F^*  \leq \left(1 - \frac{ \mu_2 }{ L_2}\right) \left(1 - \frac{ \mu_1 }{ L_1}\right) [ F(x^{k}) - F^*].
\end{equation}
\end{theorem}
\begin{proof}

By using Lipschitz continuity of the gradient of $f$ we have
\begin{align*} \label{prox-lip}
    F(T_{L_1}(x^{k})) & = F(T^1_{L_1}(x^{k}), x_2^k)) = f(T_{L_1}^1(x^{k}), x_2^k) + g_1(T_{L_1}^1(x^k)) + g_2(x_2^{k}) \nonumber\\
    &= f(T_{L_1}^1(x^{k}),x_2^k) + g_1(T_{L_1}^1(x^k)) + g_2(x_2^{k}) + g_1(x_1^k) - g_1(x_1^k)
    \\
    &\leq F(x^k) + \langle \nabla_1 f(x^k) , T_{L_1}^1(x^{k})-x_1^k \rangle + \frac{{L_1}}{2}||T_{L_1}^1(x^{k})-x_1^k||^2 \\
    &+ g_1(T_{L_1}^1(x^{k})) - g_1(x_1^k) \nonumber \\
	& \leq F(x^k) - \frac{1 }{ 2L_1} \mathcal{D}^1(x^k,{L_1}) \leq F(x^k) - \frac{ \mu_1}{{L_1}}[F(x^k) - F^*],
\end{align*}
which uses the definition of $T_M(x_{k})$ and $\mathcal{D}_1$ followed by the proximal-PL inequality~\eqref{prox-pl}. This subsequently implies that 
\begin{equation*}
    F(x^{k+\frac{1}{2}}) - F^*  \leq F(T_{L_1}(x^{k})) - F^* \leq  \left(1 - \frac{ \mu_1 }{L_1}\right) [ F(x_{k}) - F^*],
\end{equation*}
The same derivation for the point $x^{k+\frac{1}{2}}$ gives
\begin{equation*}
    F(x^{k+1}) - F^*  \leq F(T_{L_2}(x^{k+\frac{1}{2}})) - F^* \leq  \left(1 - \frac{ \mu_2 }{ L_2}\right) [ F(x^{k+\frac{1}{2}}) - F^*],
\end{equation*}
as well as the result of the theorem.
\end{proof}

Notice that above derivation does not require specification of what norm is used, so the above theorem guarantees that alternating minimization is better than the gradient methods w.r.t. any norm. In other words, alternating minimization pick up the geometric structure of the problem automatically and convergence rate of AM algorithm is not worse than the convergence rate  of the gradient method in the basis with the best possible condition number.

\subsubsection{Example with $\|x\|_A = \sqrt{\la Ax, x\ra}$}
As an example we consider here a norm endowed with a matrix.
The following result can be found in the 14-th chapter of \cite{Beck.ch14} or in \cite{Nesterov2014book}
\begin{equation}
    \|G_{M_2}^2(x^{k+\frac{1}{2}})\|^2_2  \leq 2L_2\left(f(x^{k+\frac{1}{2}})- f(x^{k+1})\right)
    \label{s.d.-1}
\end{equation}
\begin{equation}
    \|G_{M_1}^1(x^{k})\|^2_2  \leq 2L_1\left(f(x^{k})- f(x^{k+\frac{1}{2}})\right)
    \label{s.d.-2}
\end{equation}
where again we suppose that constant $L_1$ and $L_2$ can be different for different blocks:
\begin{multline*}
    f(u, v) \leqslant f(\xi, \eta)
    \\
    + \langle \nabla_1 f(\xi,\eta), u-\xi \rangle + \langle \nabla_2 f(\xi, \eta), v-\eta \rangle +  \frac{L_1}{2} \|u-\xi\|^2_2 + \frac{L_2}{2} \|v-\eta\|^2_2,
\end{multline*}
and the the constant in the regular definition of Lipshitz continuity of the gradient of $f$ is described by $L = \max(L_1, L_2)$.

Let also consider a norn endowed with a matrix:
\[\|x\|_A^2 = \la Ax, x\ra.\]
We can guarantee that the following holds for any matrix $A$
\begin{equation*}
    \|\nabla_2f(x^{k+\frac{1}{2}})\|^2_{A^{-1}}  \leq 2L^A_2\left(f(x^{k+\frac{1}{2}})- f(x^{k+1})\right)
\end{equation*}
\begin{equation*}
    \|\nabla_1f(x^{k})\|^2_{A^{-1}}  \leq 2L^A_1\left(f(x^{k})- f(x^{k+\frac{1}{2}})\right)
\end{equation*}
where

Let suppose that PL-conditions can be satisfied in the other basis
\begin{equation*}
    \mu_1^{B_1}\left(f(x^{k+\frac{1}{2}})- f(x^*)\right) \leq \|\nabla_1f(x^{k})\|^2_{B_1^{-1}}
\end{equation*}
\begin{equation*}
    \mu_2^{B_2}\left(f(x^{k+1})- f(x^*)\right) \leq \|\nabla_2f(x^{k+\frac{1}{2}})\|^2_{B_2^{-1}},
\end{equation*}
for all $B_1 \in \mathbf{B_1}$ and $B_2 \in \mathbf{B_2}$. Then
\[
    f(x^{k+1})- f(x^*) \leq \min_{B_2 \in \mathbf{B_2}} \left(1 - \frac{\mu_2^{B_2}}{L_2^{B_2}}\right)\times\min_{B_1 \in \mathbf{B_1}}\left(1 - \frac{\mu_1^{B_1}}{L_1^{B_1}}\right)\times \left(f(x^{k})- f(x^{*})\right).
\]

%---------------------------------------------------------------------------------------------------------------------------
\subsection{Polynomial convergence}
Our analysis mainly relies on the fact that alternating step is not worse than any step of any method w.r.t the only block of variables, e.g.
\[ 
 f(x_1^{k+1}, x_2^k) = \min_{z \in Q_1} f(z, x_2^k) \leq f\left(\Step(x^k), x_2^k\right),
\]
since $x_1^{k+1} = \argmin_{z \in Q_1} f(z, x_2^k).$

In particular if $\Step(x)$ defined as gradient step w.r.t. $p$ norm
\[\Step^1(x) = \argmin_{u \in \R^{n_1}} f(x) + \la \nabla_1 f(x), u - x_1 \ra + \frac{L_p}{2}\|u-x_1\|^2_{p},\]
and
\[\Step^2(x) = \argmin_{v \in \R^{n_2}} f(x) + \la \nabla_2 f(x), v - x_2 \ra + \frac{L_p}{2}\|v-x_2\|^2_{p},\]
\cite{2013arXiv1304.2338K} guarantee that 
\[f(x^k) - f(\Step(x^{k})) \geq \frac{1}{2L_p}\|\nabla f(x^k)\|^2_{p*},\]
and
\[f(x^N) -f* \lesssim \frac{L_pR_p^2}{N},\]
where
\[R_p^2 = \max_{x: f(x) \leq f(x_0)}\|x-x*\|_p.\]
So for alternating minimization we can guarantee
\[f(x^N) -f* \leq \min_{p\in[1,\infty]}\frac{2L_pR_p^2}{N}\]

\section{Accelerated Alternating Minimization}
In this section we  describe accelerated method for alternating minimization, which is originates in \cite{nesterov2018primaldual}.
But before notice, that algorithm \ref{alg-1} does not use the constant of strong convexity and consequently adapts to strong convexity of the problem. If the problem is non-strongly convex or PL condition is not satisfied the  algorithm \ref{alg-1} possesses the following convergence rate 
\[f(x^N) - f_{opt} \leq \max \left\{ \frac{f(x_0) - f_{opt}}{2^{(N-1)/2}}, \frac{8\min(L_1,L_2)R^2}{N-1} \right\}.\]
The proof can be found in \cite{Beck.ch14}. The following algorithm requires the knowing of the parameter $\mu$ of strong convexity. But it is possible to use this method with $\mu=0$. In this case the algorithm turns exactly into algorithm 1 from \cite{2019arXiv190603622G}. The other interesting result that in the case method started with $\mu=0$ method automatically adapts to strong convexity of the problem and poses at least the same linear convergence rate as a gradient descent (see Lemma \ref{non-acc-conv}).

The set $\{1,\ldots, m\}$ of indices of the orthonormal basis vectors $\{e_i\}_{i=1}^m$ is divided into $n$ disjoint subsets (blocks) $I_k$, $k\in\{1,\ldots,n\}$. Let $S_k(x)=x+\spn\{e_i:\ i\in I_k\}$, i.e. the affine subspace containing $x$ and all the points differing from $x$ only over the block $k$. We use $x_i$ to denote the components of $x$ corresponding to the block $i$ and $\nabla_i f(x)$ to denote the gradient corresponding to the block $i$. We will further require that for any $k\in\{1,\ldots,n\}$ and any $z\in\mathbb{R}^m$ the problem $f(x)\to\min\limits_{x\in S_i(z)}$ has a solution, and this solution is easily computable.

\begin{algorithm}[H]
\caption{Accelerated Alternating Minimization (AAM)}
\label{AAM-2}

\begin{algorithmic}[1]
   \REQUIRE Starting point $x_0$
    \ENSURE $x^k$
   \STATE Set $A_0=0$, $x^0 = v^0$, $\tau_0 = 1$
   \FOR{$k \geqslant 0$}
	\STATE Set 
	\begin{equation}
	    \beta_k = \argmin\limits_{\beta\in [0,1]} f\left(x^k + \beta (v^k - x^k)\right)\label{line-search}
	\end{equation}
	\STATE Set $y^k = x^k + \beta_k (v^k - x^k)\quad $\COMMENT{Extrapolation step}
    \STATE Choose $i_k=\argmax\limits_{i\in\{1,\ldots,n\}} \|\nabla_i f(y^k)\|_2$
\STATE Set $x^{k+1}=\argmin\limits_{x\in S_{i_k}(y^k)} f(x)$\quad \COMMENT{Block minimization}
\STATE 
 If L is known choose $a_{k+1}$ s.t. 
$\frac{a_{k+1}^2}{(A_{k}+a_{k+1})(\tau_{k}+\mu a_{k+1})} = \frac{1}{Ln}$
\\
If L is unknown, find largest $a_{k+1}$ from the equation 
\begin{multline}
    f(y^k) - \frac{a_{k+1}^2}{2(A_{k} + a_{k+1})(\tau_{k}+\mu a_{k+1})}\| \nabla f(y^k) \|_2^2 +
    \\
    \frac{\mu \tau_k a_{k+1}}{2(A_{k} + a_{k+1})(\tau_{k}+\mu a_{k+1})}\| v^k - y^k \|_2^2 = f(x^{k+1})
    \label{aam-s.d.}
\end{multline}

\STATE  Set $A_{k+1} = A_{k} + a_{k+1}$, $\tau_{k+1} = \tau_k + \mu a_{k+1}$
%\STATE  Set $\psi_{k+1}(x) =  \psi_{k}(x) + a_{k+1}\{f(y^k) + \langle \nabla f(y^k), x - y^k \rangle\}$
\STATE Set $v^{k+1} = v^{k}-a_{k+1}\nabla f(y^k)$. \COMMENT{Update momentum term}
%\STATE $k = k + 1$
\ENDFOR
\end{algorithmic}

\end{algorithm}

We will begin with one key Lemma. Let us introduce an auxiliary functional sequence defined as \[\psi_0(x)=\frac{1}{2}\|x-x^0\|_2^2,\]

\[\psi_{k+1}(x) = \psi_{k}(x) + a_{k+1}\{f(y^k) + \langle \nabla f(y^k), x - y^k \rangle\ + \frac{\mu}{2}\| x - y^k \|_2^2\}.\]

For 
\[ l_k(x) = \sum_{i=0}^k a_{i+1} \{ f(y^i) + \langle \nabla f(y^i), x - y^i \rangle\ + \frac{\mu}{2}\| x - y^i \|_2^2 \} \]
we can write
\[
\psi_{k+1}(x) = \psi_{0}(x) + l_k(x)
\]
It is easy to see that $\psi_k(x)$ is $\tau_k$ strongly convex function with \[\tau_k = 1 + \mu\sum_{i=0}^k a_i = 1 + \mu A_k.\]
\begin{lemma} 
\label{AAM-2_Ak_rate}
After $k$ steps of Algorithm \ref{AAM-2} it holds that

\begin{equation}
    \label{eq:main_recurrence}
    A_{k}f(x^{k}) \leqslant \min_{x \in \mathbb{R}^m} \psi_{k}(x) = \psi_{k}(v^{k}).
\end{equation}
Moreover, if the objective is $L$-smooth and $\mu$-strongly convex 
\[A_k \geqslant \max \left\{\frac{k^2}{4Ln}, \frac{1}{nL}\left(1 - \sqrt{\frac{\mu}{nL}}\right)^{-k-1}\right\},\]
where $n$ is the number of blocks.
\end{lemma}
\begin{proof}

First, we prove inequality \eqref{eq:main_recurrence} by induction over $k$. For $k=0$, the inequality holds. Assume that 
\[A_{k}f(x^{k}) \leqslant \min_{x \in \mathbb{R}^m} \psi_{k}(x) = \psi_{k}(v^{k}).\]
Then
\begin{multline*}
    \psi_{k+1}(v^{k+1}) = \min_{x \in \mathbb{R}^m} \Bigg\{ \psi_{k}(x) + a_{k+1}\{f(y^k) + \langle \nabla f(y^k), x - y^k \rangle + \frac{\mu}{2}\| x - y^k \|_2^2\} \Bigg\}
    \\
    \geqslant \min_{x \in \mathbb{R}^m} \Bigg\{ \psi_{k}(v^k) + \frac{\tau_k}{2}\| x - v^k \|_2^2 + a_{k+1}\{f(y^k) + \langle \nabla f(y^k), x - y^k \rangle 
    \\
    \shoveright{+ \frac{\mu}{2}\| x - y^k \|_2^2\} \Bigg\}}
    \\
    \geqslant \min_{x \in \mathbb{R}^m} \Bigg\{ A_{k}f(x^k) + \frac{\tau_k}{2}\| x - v^k \|_2^2 + a_{k+1}\{f(y^k) + \langle \nabla f(y^k), x - y^k \rangle 
    \\
    + \frac{\mu}{2}\| x - y^k \|_2^2\} \Bigg\}
\end{multline*}
Here we used that $\psi_{k}$ is a strongly convex function with minimum at $v^k$ and that $f(y^k)\leqslant f(x^k)$. 

By the optimality conditions for the problem $\min\limits_{\beta\in [0,1]} f\left(x^k + \beta (v^k - x^k)\right)$, either
\begin{enumerate}
    \item $\beta_k = 1$, $\langle \nabla f(y^k),x^k - v^k \rangle \geqslant 0$, $y^k = v^k$;
    \item $\beta_k \in (0,1)$ and $\langle \nabla f(y^k),x^k - v^k \rangle = 0$, $y^k = v^k + \beta_k (x^k - v^k)$;
    \item $\beta_k = 0$ and $\langle \nabla f(y^k),x^k - v^k \rangle \leqslant 0$, $y^k = x^k$ .
\end{enumerate}In all three cases,  $\langle \nabla f(y^k), v^k - y^k \rangle \geqslant 0$.

Thus 
\begin{multline*}
    \psi_{k+1}(v^{k+1}) \geqslant  \min_{x \in \mathbb{R}^m} \Big\{ A_{k}f(y^k) + \frac{\tau_k}{2}\| x - v^k \|_2^2 + a_{k+1}\{f(y^k) + \langle \nabla f(y^k), x - y^k \rangle
    \\
    + \frac{\mu}{2}\| x - y^k \|_2^2\} \Big\}
\end{multline*}
The explicit solution to the above quadratic optimization problem is
\[
x=\frac{1}{\tau_{k+1}}(\tau_k v^k + \mu a_{k+1}y^k - a_{k+1} \nabla f(y^k))
\]

By plugging in the solution and using $\langle \nabla f(y^k), v^k - y^k \rangle \geqslant 0$, we obtain

\begin{align*}
    \psi_{k+1}(v^{k+1}) &\geqslant A_{k+1}f(y^k) - \frac{a_{k+1}^2}{2\tau_{k+1}}\| \nabla f(y^k) \|_2^2 + \frac{\mu \tau_k a_{k+1}}{2\tau_{k+1}}\| v^k - y^k \|_2^2.
\end{align*}

Our next goal is to show that 

\begin{align*}
    A_{k+1}f(y^k) - \frac{a_{k+1}^2}{2\tau_{k+1}}\| \nabla f(y^k) \|_2^2 + \frac{\mu \tau_k a_{k+1}}{2\tau_{k+1}}\| v^k - y^k \|_2^2 &\geqslant A_{k+1}f(x^{k+1})
\end{align*}

which proves the induction step.

To do this, by the $L$-smoothness of the objective, we have $\forall i$
\[
f(y^k)-\frac{1}{2L}\|\nabla_i f(y^k)\|_2^2\geqslant f(x_i^{k+1}),
\]
where $x_i^{k+1}=\argmin_{x\in S_{i}} f(x)$. Since $i_k=\argmax_{i} \|\nabla_i f(y^k)\|_2^2$, $$\|\nabla_{i_k} f(y^k)\|_2^2\geqslant \frac{1}{n}\|\nabla f(y^k)\|_2^2$$ 
and
$$f(y^k)-\frac{1}{2Ln}\|\nabla f(y^k)\|_2^2\geqslant f(y^k)-\frac{1}{2L}\|\nabla_{i_k} f(y^k)\|_2^2\geqslant f(x^{k+1}),$$
Choosing $a_{k+1}$ such that $\frac{a_{k+1}^2}{2A_{k+1}\tau_{k+1}}\geqslant \frac{1}{2Ln}$ implies 
%\frac{a_{k+1}^2}{2A_{k+1}\tau_{k+1}}\geqslant \frac{1}{2Ln}

\begin{multline*}
    A_{k+1}f(y^k) - \frac{a_{k+1}^2}{2\tau_{k+1}}\| \nabla f(y^k) \|_2^2 + \frac{\mu \tau_k a_{k+1}}{2\tau_{k+1}}\| v^k - y^k \|_2^2 
    \\
    \geqslant A_{k+1}f(y^k) - \frac{a_{k+1}^2}{2\tau_{k+1}}\| \nabla f(y^k) \|_2^2  
    \geqslant A_{k+1}f(y^k) - \frac{A_{k+1}}{2Ln} \| \nabla f(y^k) \|_2^2     
    \\
    \geqslant A_{k+1}f(x^{k+1})
\end{multline*}
which proves the induction step.

Rewriting the rule for choosing $a_{k+1}$ gives $\frac{a_{k+1}^2}{(A_{k}+a_{k+1})(\tau_{k}+\mu a_{k+1})}\geqslant \frac{1}{Ln}$.

Let us estimate the rate of the growth for $A_k$. $\tau_k = 1 + \mu\sum_{i=0}^k a_i = 1 + \mu A_k$. $\frac{a_{k+1}^2}{2A_{k+1}\tau_{k+1}}\geqslant \frac{1}{2Ln}$
\begin{align*}
	a^2_{k} & \geqslant 
	\frac{A_k \tau_k}{{nL}} = \frac{{A_k + \mu A_k^2}}{nL}
\end{align*}

\begin{align}
	a_{k} & \geqslant 
	\frac{1}{\sqrt{nL}} \sqrt{A_k + \mu A_k^2}
	\geqslant 
	\sqrt{\frac{\mu}{2Ln}}A_k
\end{align}

\begin{equation*}
    \sqrt{A_i} - \sqrt{A_{i-1}}
    \geqslant \frac{A_i-A_{i-1}}{\sqrt{A_i} + \sqrt{A_{i-1}}} 
    \geqslant \frac{a_i}{2\sqrt{A_i}}
    \geqslant \frac{\sqrt{1+\mu A_i}}{2\sqrt{Ln}}
\end{equation*}

Summing it up for $i=1,\dots,k$ we get \[A_k \geqslant \frac{k^2}{4Ln}\]

We also have \[A_{k+1} = A_{k} + a_{k+1} \geqslant A_k + \sqrt{\frac{\mu}{nL}}A_{k+1}\]
which leads to \[A_{k+1} \geqslant \left(1 - \sqrt{\frac{\mu}{nL}}\right)^{-1}A_{k}\]

To use this bound we only need to estimate A1 , which we can do as follows:
\[A_1 = \frac{a_1^2}{A_1} \geqslant  \frac{a_1^2} {(1 + \mu A_1)A1} \geqslant  \frac{a_1^2}{A_1 \tau_1} \geqslant \frac{1}{nL}\]

By recursively applying the last bound we reach the desired result:
\[A_k \geqslant \max \left\{\frac{k^2}{4Ln}, \frac{1}{nL}\left(1 - \sqrt{\frac{\mu}{nL}}\right)^{-k+1}\right\}\]

\end{proof}
\begin{theorem} 
After $k$ steps of Algorithm \ref{AAM-2} it holds that

\begin{equation}
    \label{eq:main_result}
    f(x^k) - f(x_*) \leqslant nLR^2\min \left\{ \frac{4}{k^2}, \left(1 - \sqrt{\frac{\mu}{nL}}\right)^{k-1}\right\}
\end{equation}
\end{theorem}

\begin{proof}
From the convexity of $f(x)$ we have
\begin{multline*}l_k(x_*) = \sum_{i=0}^{k} a_{i+1}(f(y^i)+\langle\nabla f(y^i),x_*-y^i\rangle + \frac{\mu}{2}\| x_* - y^i \|_2^2) 
\leqslant A_{k+1} f(x_*).
\end{multline*}
From Lemma~\eqref{AAM-2_Ak_rate} we have 
\begin{multline*}
    A_k f(x^k)\leqslant\psi_{k}(v^k)
    \leqslant\psi_k(x_*)=\frac{1}{2}\|x_*-x^0\|_2^2 \notag \\
    +\sum_{i=0}^{k-1} a_{i+1}(f(y^i)+\langle\nabla f(y^i),x_*-y^i\rangle + \frac{\mu}{2}\| x_* - y^i \|_2^2)
    \leqslant 
     A_k f(x_*)+\frac{1}{2}\|x_*-x^0\|_2^2
\end{multline*}

\begin{equation*}
    f(x^k) - f(x_*) \leqslant \frac{R^2}{2A_k} \leqslant nLR^2\min \left\{ \frac{4}{k^2}, \left(1 - \sqrt{\frac{\mu}{nL}}\right)^{k-1}\right\}.
\end{equation*}
\end{proof}

The other observation explains behaviour of this method when $\mu$ is unknown.
\begin{lemma}
    Algorithm \ref{AAM-2} started with $\mu=0$ automatically adapts to strong convexity of the problem and has linear convergence:
    \[ f(x^{k+1}) - f(x^*) \leq \Pi_{i=0}^{k-1}  \left(1 - \frac{\mu}{\hat L_i}\right) \cdot (f(x^{0})- f(x^{*})),\]
    where $\hat L_i = \frac{A_{i} + a_{i+1}}{a_{i+1}^2}$ is the upper bound on $L$ at the $i$-th iteration.
    \label{non-acc-conv}
\end{lemma}
\begin{proof}
    (\ref{aam-s.d.}) with $\mu=0$ implies sufficient decrease result:
    \begin{equation*}
        f(y^k) - \frac{a_{k+1}^2}{2(A_{k} + a_{k+1})} \| \nabla f(y^k) \|_2^2  = f(x^{k+1}) \geq f(y^{k+1})
    \end{equation*}
    since (\ref{line-search}) implies that $f(y^k)\leq f(x^k)$.
    By combining this result with PL-condition
    \begin{equation*}
         \|\nabla f(y^k)\|_2^2  \geq 2\mu\left(F(y^{k})- F(x^*)\right)
    \end{equation*}
    we have linear convergence
    \begin{multline*}
        \left(f(y^{k+1})- f(x^*)\right) \leq \left(1 - \frac{\mu a_{k+1}^2}{A_{k} + a_{k+1}}\right) \left(f(y^{k})- f(x^{*})\right) 
        \\
        \leq \Pi_{i=0}^{k}  \left(1 - \frac{\mu a_{i+1}^2}{A_{i} + a_{i+1}}\right) \left(f(x^{0})- f(x^{*})\right)
    \end{multline*}
    And finally block minimization step guarantees that $f(x^{k+1})\leq f(y^k)$, so we have
    \[ f(x^{k+1}) - f(x^*) \leq \Pi_{i=0}^{k-1}  \left(1 - \frac{\mu a_{i+1}^2}{A_{i} + a_{i+1}}\right) \left(f(x^{0})- f(x^{*})\right)\]
\end{proof}

\section{Application}
Consider the following problem of minimazing a quadratic function
\begin{equation}
    f(z) = \|Wz-b\|_2^2 \rightarrow \min_z
    \label{qprob}
\end{equation}
this is a strongly convex function with $\mu = \lambda_{\min}(W^T W).$

This problem can be solved with algorithm \ref{alg-1} by splitting the vector variable $z$ into two vector variables with the dimension:
\[
    z = 
    \begin{pmatrix}
        x
        \\
        y
    \end{pmatrix}.
\]
Then split matrix $W$ into four blocks with the same size
\[
    W = 
    \begin{pmatrix}
        A B
        \\
        C D
    \end{pmatrix}.
\]
and vector $b$
\[
    b = 
    \begin{pmatrix}
        d
        \\
        c
    \end{pmatrix}.
\]
The equivalent problem to (\ref{qprob}) is
\[ \|Ax + By - c\|_2^2 + \| Cx + Dy -d\|_2^2 \rightarrow \min_{x,y}\]
and the iterations of the algorithm \ref{alg-1} can be written explicitly
\begin{align*}
    x^{k+1} &= (A^T A + C^T C)^{-1}\big[A^T(c-By^k) + C^T(d-Dy^k)\big]
    \\
    y^{k+1} &= (B^T B + D^T D)^{-1}\big[B^T(c-Ax^k) + D^T(d-Cx^k)\big]
\end{align*}

Next we provide comparison between Algorithm \ref{alg-1}, Algorithm \ref{AAM-2} started with $\mu=0$ and $\mu=\mu^*$, and the following accelerated algorithm
\begin{algorithm}[H]
\caption{Fast Gradient Method (FGM)}

\begin{algorithmic}[1]
   \REQUIRE Starting point $z_0$.
    \ENSURE $z^k$
    \STATE Set $v^0 = z^0$.
    \FOR{$k \geqslant 0$}
        \STATE $z^{k+1} = v^k - \frac{1}{L} \nabla f(v^k)$
        \STATE $v^{k+1} = z^k + \frac{k}{k+3} (z^{k+1}-z^k)$
\ENDFOR
\end{algorithmic}
\end{algorithm}

\begin{figure}
\begin{center}
\includegraphics[width=8.4cm]{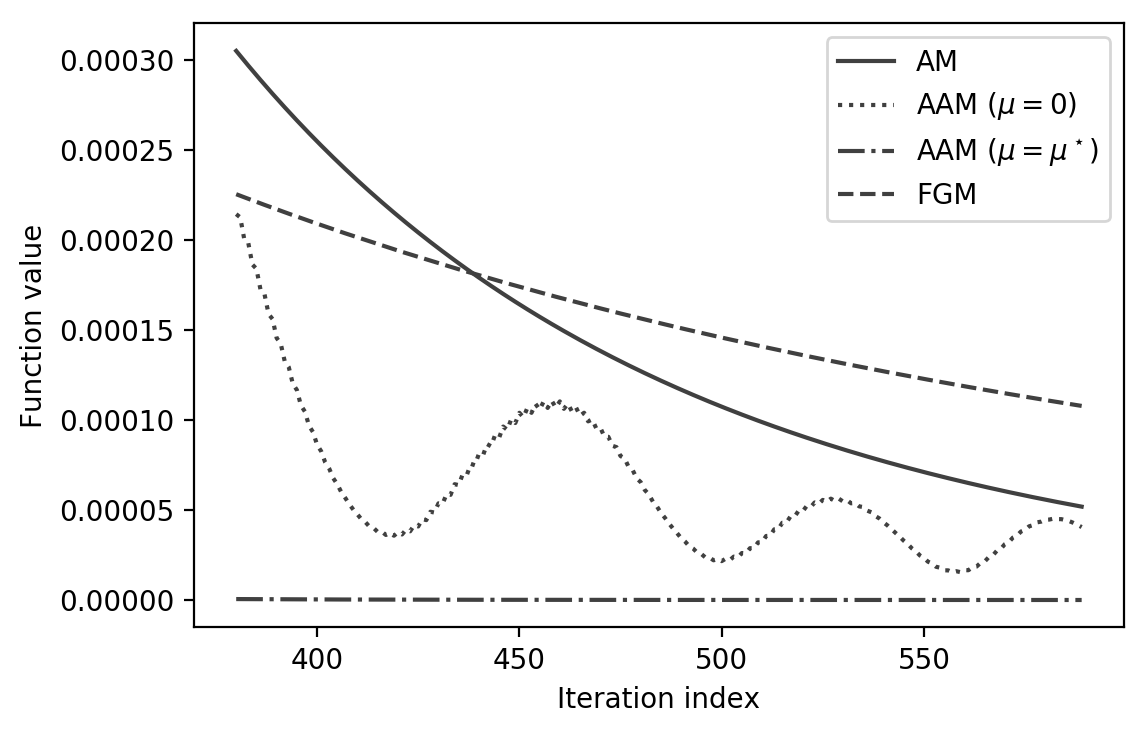}    % The printed column width is 8.4 cm.
\caption{Comparison for quadratic function} 
\label{fig-init}
\end{center}
\end{figure}

\section{Linear convergence under general convex constraint sets}
%\pd{What is the difference with Section 2? It seems that the problem is the same.}
The proof of linear convergence relies on Lemma~\ref{lem:1}, which requires special structure of constraint sets~\eqref{sets-structure}. For general convex constraints, we was able to prove only a weaker result, which is presented in this section. 

The following lemma is used instead of Theorem~\ref{th:linear} and Lemma~\ref{lem:1}.

\begin{lemma} Strong convexity of $f$ implies "nearly" PL-condition:
\begin{equation}
    \mu_1\left(F(x^{k+\frac{1}{2}})- F(x^*)\right) \leq \frac{1}{2}\|G_{M_1}^1(x^{k})\|^2_2
\end{equation}
\begin{equation}
    \mu_2\left(F(x^{k+1})- F(x^*)\right) \leq \frac{1}{2} \|G_{M_2}^2(x^{k+\frac{1}{2}})\|^2_2
\end{equation}
\end{lemma}   

\begin{proof}
Consider 
\[
    T_M^2(x^{k+\frac{1}{2}}) = \argmin_{u \in Q}\left(g_2(u) + \frac{{M_2}}{2} \|u-(x^k_2 - \frac{1}{M_2}\nabla f(x^{k+\frac{1}{2}}))\|^2_2\right)
\]

Since $T_{M_2}^2(x^{k+\frac{1}{2}})$ is a minimizer, optimality condition gives for all $v \in Q$

\begin{equation*}
    \la \partial g_2(T_{M_2}^2(x^{k+\frac{1}{2}})) + \nabla_2 f(x^{k+\frac{1}{2}}) + M(T_{M_2}^2(x^{k+\frac{1}{2}}) - x^k_2), v - T_{M_2}^2(x^{k+\frac{1}{2}}) \ra \geq 0
\end{equation*}
or
\begin{multline}
    \la \nabla_2 f(x^{k+\frac{1}{2}}) , v - T_{M_2}^2(x^{k+\frac{1}{2}}) \ra
    \geq \la -\partial g_2(T_{M_2}^2(x^{k+\frac{1}{2}})), v - T_{M_2}^2(x^{k+\frac{1}{2}}) \ra
    \\
    + \la G^2_{M_2}(x^{k+\frac{1}{2}}), v - T_{M_2}^2(x^{k+\frac{1}{2}}) \ra
    \label{opt-2.2}
\end{multline}
    
Since $f$ is strongly convex
\begin{multline}
    f(u, v) \geq f(x^{k+1}_1, x^k_2) + \la \nabla_1 f(x^{k+1}_1, x^k_2), u-x^{k+1}_1 \ra + \la \nabla_2 f(x^{k+1}_1, x^k_2), v-x^k_2 \ra 
    \\
    +  \frac{\mu_1}{2} \|u-x^{k+1}_1\|_2^2 + \frac{\mu_2}{2} \|v-x^k_2\|_2^2 \geq
    \\
    \stackrel{\scriptsize{\circled{1}}}{\geq} f(x^{k+1}_1, x^k_2) - \la \partial g_1(x_1^{k+1}), u-x^{k+1}_1 \ra + \la \nabla_2 f(x^{k+1}_1, x^k_2), v-x^k_2 \ra + \frac{\mu_2}{2} \|v-x^k_2\|_2^2 
    \\
    \stackrel{\scriptsize{\circled{2}}}{\geq}
    f(x^{k+1}_1, x^k_2) + g_1(x_1^{k+1})-g_1(u) + \la \nabla_2 f(x^{k+1}_1, x^k_2), v-x^k_2 \ra + \frac{\mu_2}{2} \|v-x^k_2\|_2^2 
    \\
    \stackrel{\scriptsize{\circled{3}}}{\geq} f(x^{k+1}_1, x^k_2) + g_1(x_1^{k+1}) + \la \nabla_2 f(x^{k+1}_1, x^k_2), T_{M_2}^2(x^{k+\frac{1}{2}})-x^k_2 \ra + \frac{\mu_2}{2} \|v-x^k_2\|_2^2
    \\
    - g_1(u)  - \la \partial g_2(T_{M_2}^2(x^{k+\frac{1}{2}})), v - T_{M_2}^2(x^{k+\frac{1}{2}}) \ra + \la G^2_{M_2}(x^{k+\frac{1}{2}}), v - T_{M_2}^2(x^{k+\frac{1}{2}}) \ra 
    \\
    \stackrel{\scriptsize{\circled{4}}}{\geq} f(x^{k+1}_1, x^k_2) + g_1(x_1^{k+1})-g_1(u) + \la \nabla_2 f(x^{k+1}_1, x^k_2), T_{M_2}^2(x^{k+\frac{1}{2}})-x^k_2 \ra - g_2(v) 
    \\
    + g_2(T_{M_2}^2(x^{k+\frac{1}{2}}))  + \frac{\mu_2}{2} \|v-x^k_2\|_2^2 + \la G^2_{M_2}(x^{k+\frac{1}{2}}), v - T_{M_2}^2(x^{k+\frac{1}{2}}) \ra 
    \\
    \stackrel{\scriptsize{\circled{5}}}{\geq}f(x^{k+1}_1, T_{M_2}^2(x^{k+\frac{1}{2}})) - \frac{M_2}{2} \|\frac{-1}{M_2}G_{M_2}^2(x^{k+\frac{1}{2}})\|^2_2 + g_1(x_1^{k+1})-g_1(u) - g_2(v)
    \\
    + g_2(T_{M_2}^2(x^{k+\frac{1}{2}})) + \frac{\mu_2}{2} \|v-x^k_2\|_2^2 + \la G^2_{M_2}(x^{k+\frac{1}{2}}), v - T_{M_2}^2(x^{k+\frac{1}{2}}) \ra
    \label{chain1.2}
\end{multline}
where we used:
\begin{itemize}
    \item \circled{1} by (\ref{opt-0.1})
    \item \circled{2} by convexity $- \la \partial g_1(x_1^{k+1}), u-x^{k+1}_1 \ra\ \geq g_1(x_1^{k+1})-g_1(u)$
    \item \circled{3} by (\ref{opt-2.2})
    \item \circled{4} by convexity $-\la \partial g_2(T_{M_2}^2(x^{k+\frac{1}{2}})), v - T_{M_2}^2(x^{k+\frac{1}{2}}) \geq g_2(T_{M_2}^2(x^{k+\frac{1}{2}})) - g_2(v)$
    \item \circled{5} since $\nabla_2f$ is $L_2$-Lipschitz continuous, for $M_2 \geq L_2$ the following holds
    \begin{multline*}
        f(x^{k+\frac{1}{2}}) + \la \nabla_2 f(x^{k+\frac{1}{2}}), T_{M_2}^2(x^{k+\frac{1}{2}}) - x^k_2\ra \geq
        \\
        \geq f(x^{k+1}_1, T_{M_2}^2(x^{k+\frac{1}{2}})) - \frac{M_2}{2} \|\frac{-1}{M_2}G_{M_2}^2(x^{k+\frac{1}{2}})\|^2_2
    \end{multline*}
\end{itemize}

Above inequality gives
\begin{multline}
    F(u, v) \geq F(x^{k+1}_1, T_{M_2}^2(x^{k+\frac{1}{2}}))
    \\
    + \frac{M_2}{2} \|\frac{-1}{M_2}G_{M_2}^2(x^{k+\frac{1}{2}})\|^2_2 +  \frac{\mu_2}{2} \|v-x^k_2\|_2^2 + \la G^2_{M_2}(x^{k+\frac{1}{2}}), v - x_2^k) \ra
    \\
    \geq F(x^{k+1}_1, T_{M_2}^2(x^{k+\frac{1}{2}}))  +  \frac{\mu_2}{2} \|v-x^k_2\|_2^2 + \la G^2_{M_2}(x^{k+\frac{1}{2}}), v - x_2^k) \ra \longrightarrow \min_{v\in \R^n}
\end{multline}

Plugging in $(u, v)=x^*$ we get one of the desired inequalities:
\begin{multline}
    F(x^*) \geq F(x^{k+1}_1, T_{M_2}^2(x^{k+\frac{1}{2}}))- \frac{1}{2\mu_2} \|G_{M_2}^2(x^{k+\frac{1}{2}})\|^2_2
    \\
    \geq F(x^{k+1})- \frac{1}{2\mu_2} \|G_{M_2}^2(x^{k+\frac{1}{2}})\|^2_2
\end{multline}
\begin{equation}
    \|G_{M_1}^1(x^{k})\|^2_2  \geq 2\mu_1\left(F(x^{k+\frac{1}{2}})- F(x^*)\right)
\end{equation}

The other inequality can be obtained the same way for the point $x^k$:
\begin{equation}
    \|G_{M_2}^2(x^{k+\frac{1}{2}})\|^2_2  \geq 2\mu_2\left(F(x^{k+1})- F(x^*)\right)
\end{equation}
\end{proof}

Combining the result of the above lemma with \eqref{s.d.-1}, \eqref{s.d.-2}, we obtain convergence rate:
\begin{align}
    \mu_1\left(F(x^{k+\frac{1}{2}})- F(x^*)\right) &\leq L_1\left(F(x^{k})- F(x^{k+\frac{1}{2}})\right)
    \\
    \mu_2\left(F(x^{k+1})- F(x^*)\right) &\leq L_2\left(F(x^{k+\frac{1}{2}})- F(x^{k+1})\right)
\end{align}

\begin{multline*}
    \left(F(x^{k+1})- F(x^*)\right) \leq (1 - \frac{\mu_2}{L_2+\mu_2}) \left(F(x^{k+\frac{1}{2}})- F(x^{*})\right)
    \\
    \left(F(x^{k+\frac{1}{2}})- F(x^*)\right) \leq (1 - \frac{\mu_1}{L_1+\mu_1})\left(F(x^{k})- F(x^{*})\right)
\end{multline*}

\begin{multline}
    \left(F(x^{k+1})- F(x^*)\right) \leq 
    \\
    \leq (1 - \frac{\mu_2}{L_2+\mu_2})(1 - \frac{\mu_1}{L_1+\mu_1}) \left(F(x^{k})- F(x^{*})\right)
\end{multline}

\bibliography{references}
\end{document}